\UseRawInputEncoding
\documentclass[12pt]{article}

\usepackage{dsfont}
\usepackage{amsfonts,amsmath,amsthm}
\usepackage{algorithm, algorithmic}
\usepackage{color}
\usepackage{graphicx}
\usepackage{stmaryrd}        % provides \llbracket and \rrbracket

\usepackage{subcaption}
\usepackage{multicol}
\usepackage{multirow}
\usepackage{url}
\usepackage[paper=a4paper,dvips,top=2cm,left=2cm,right=2cm,
foot=1cm,bottom=4cm]{geometry}

\newcommand{\ve}{{\bf e}}
\newcommand{\cset}{\mathbb C}
\newcommand{\gset}{\mathbb G}
\newcommand{\rset}{\mathbb R}
\newcommand{\qset}{\mathbb Q}
\newcommand{\dcset}{\hat{\mathbb C}}
\newcommand{\dgset}{\hat{\mathbb G}}
\newcommand{\drset}{\hat{\mathbb R}}
\newcommand{\dqset}{\hat{\mathbb Q}}
\newcommand{\vdq}[1]{\mathbf{#1}}

\usepackage{booktabs}
\begin{document}
	\large
	
	\title{{Eigenvalues of Dual Hermitian Matrices with Application in Formation Control}}%{A Supplement Matrix Method for Computing Eigenvalues of Dual Hermitian Matrices with Application in Formation Control}%\footnote{This work was partially supported by Hong Kong Innovation and Technology Commission (InnoHK Project CIMDA).}}
	\author{ Liqun Qi\footnote{Department of Mathematics, School of Science, Hangzhou Dianzi University, Hangzhou 310018 China; Department of Applied Mathematics, The Hong Kong Polytechnic University, Hung Hom, Kowloon, Hong Kong
			({\tt maqilq@polyu.edu.hk}).}
		%\and \
		%Xiangke Wang\thanks{College of Mechatronics and Automation, National University of Defence Technology, Changsha, 410073, China ({\tt
		%xkwang@nudt.edu.cn}).}
		\and \
		Chunfeng Cui\footnote{LMIB of the Ministry of Education, School of Mathematical Sciences, Beihang University, Beijing 100191 China.
			({\tt chunfengcui@buaa.edu.cn}).}
		%This author's work was supported by %Natural Science Foundation of China (No. ) and .
		%}
		%\and and \
		%Ziyan Luo\footnote{Corresponding author, Department of Mathematics,
		%Beijing Jiaotong University, Beijing 100044, China. ({\tt zyluo@bjtu.edu.cn}). This author's work was supported by Beijing Natural Science Foundation (Grant No. Z190002).}
	}
	\date{\today}
	\maketitle

	\begin{abstract}
		We propose a supplement matrix method for computing eigenvalues of a dual Hermitian matrix, and discuss its application in multi-agent formation control.
		Suppose we have a ring, which can be the real field, the complex field, or the quaternion ring.
		{We study dual number symmetric matrices, dual complex Hermitian matrices and dual quaternion Hermitian matrices in a unified frame of dual Hermitian matrices.
			An} $n \times n$ dual Hermitian matrix has $n$ dual number eigenvalues.
		We define determinant, characteristic polynomial  and  supplement matrices for a dual Hermitian matrix.   Supplement matrices are Hermitian matrices in the original ring.  The standard parts of the eigenvalues of that dual Hermitian matrix are the eigenvalues of the standard part Hermitian matrix in the original ring, while the dual parts of the eigenvalues of that dual Hermitian matrix are the eigenvalues of those  {supplement} matrices.    Hence, by {applying} any practical method for computing eigenvalues of Hermitian matrices in the original ring, we have a practical method for computing eigenvalues of a dual Hermitian matrix.   We call this method the supplement matrix method.   In multi-agent formation control, a desired relative configuration scheme may be given.   People need to know if this scheme is reasonable such that a feasible solution of configurations of these multi-agents exists.   By exploring the eigenvalue problem of dual Hermitian matrices, and its link with the unit gain graph theory, we open a cross-disciplinary approach to solve the relative configuration problem.
		Numerical experiments are reported.

		\medskip

		% \medskip

		
		%\medskip
		% \textbf{AMS subject classifications. }
	\end{abstract}

	\renewcommand{\Re}{\mathds{R}}
	\newcommand{\rank}{\mathrm{rank}}
	\newcommand{\X}{\mathcal{X}}
	\newcommand{\A}{\mathcal{A}}
	\newcommand{\I}{\mathcal{I}}
	\newcommand{\B}{\mathcal{B}}
	\newcommand{\C}{\mathcal{C}}
	\newcommand{\OO}{\mathcal{O}}
	\newcommand{\e}{\mathbf{e}}
	\newcommand{\0}{\mathbf{0}}
	\newcommand{\dd}{\mathbf{d}}
	\newcommand{\ii}{\mathbf{i}}
	\newcommand{\jj}{\mathbf{j}}
	\newcommand{\kk}{\mathbf{k}}
	\newcommand{\va}{\mathbf{a}}
	\newcommand{\vb}{\mathbf{b}}
	\newcommand{\vc}{\mathbf{c}}
	\newcommand{\vq}{\mathbf{q}}
	\newcommand{\vg}{\mathbf{g}}
	\newcommand{\pr}{\vec{r}}
	\newcommand{\ps}{\vec{s}}
	\newcommand{\pt}{\vec{t}}
	\newcommand{\pu}{\vec{u}}
	\newcommand{\pv}{\vec{v}}
	\newcommand{\pw}{\vec{w}}
	\newcommand{\pp}{\vec{p}}
	\newcommand{\pq}{\vec{q}}
	\newcommand{\pl}{\vec{l}}
	\newcommand{\vt}{\rm{vec}}
	\newcommand{\vx}{\mathbf{x}}
	\newcommand{\vy}{\mathbf{y}}
	\newcommand{\vu}{\mathbf{u}}
	\newcommand{\vv}{\mathbf{v}}
	\newcommand{\y}{\mathbf{y}}
	\newcommand{\vz}{\mathbf{z}}
	\newcommand{\T}{\top}
	
	\newtheorem{Thm}{Theorem}[section]
	\newtheorem{Def}[Thm]{Definition}
	\newtheorem{Ass}[Thm]{Assumption}
	\newtheorem{Lem}[Thm]{Lemma}
	\newtheorem{Alg}[Thm]{Algorithm}
	\newtheorem{Prop}[Thm]{Proposition}
	\newtheorem{Cor}[Thm]{Corollary}
	\newtheorem{example}[Thm]{Example}
	\newtheorem{remark}[Thm]{Remark}
	
	\section{Introduction}
	
	Dual numbers, dual quaternions, their vectors and matrices, as well as their applications have a long history.   It was British mathematician William Kingdon Clifford who introduced dual numbers in 1873 \cite{Cl73}.  Then German mathematician Eduard Study introduced dual angles in 1903 \cite{St03}.   These started the study and applications of dual numbers, dual number vectors and dual number matrices in kinematics, dynamics and robotics \cite{An12, Fi98, GL87, PS07, PV09, Ud21, Wa21}.  Later, dual quaternions, in particular, unit dual quaternions, found wide applications in hand-eye calibration, neuroscience, multi-agent formation control and simultaneous location and mapping (SLAM), etc., \cite{BK20, BLH19, CKJC16,Da99,LLB13,QC24,QWL23,WYL12}.  Among these, dual quaternion matrices are used in multi-agent formation control \cite{QC24, QWL23}. Very recently, dual complex matrices found applications in brain science \cite{WDW24}.

	In 2007,  Pennestr\`i and Stefanelli \cite{PS07} proposed  {the problem of} %the issue for
	computing eigenvalues and singular values of dual number matrices.   Also see \cite{PV09}.    In 2023, Qi and Luo {\cite{QL23}}  showed that an $n \times n$ dual quaternion Hermitian matrix {has $n$ dual number eigenvalues}, and established the singular value decomposition of a dual quaternion matrix.   These also apply to dual number matrices and dual complex matrices \cite{QC23}.  The relative configuration adjacent matrices and %dual quaternion
	Laplacian matrices in multi-agent formation control are dual quaternion Hermitian matrices \cite{QC24, QWL23}.  %The adjacent and Laplacian matrices of dual complex unit gain graphs are dual complex Hermitian matrices \cite{CLQW24}.   The relation between the spectral radii of the adjacent and Laplacian matrices of a dual complex unit gain graph, and the spectral radii  of the adjacent and Laplacian matrices of the underlying graph was discussed in \cite{CLQW24}.
	
	Then, several numerical methods for computing eigenvalues of dual quaternion Hermitian matrices arose.  These include a power method \cite{CQ23}, a bidiagonalization method \cite{DLWW24} and a Rayleigh quotient iteration method \cite{DWD24}.
	
	%Even real matrices may have complex eigenvalues.   Thus, it is better to study these in the scope of dual complex matrices, by regarding the set of dual complex numbers as a commutative ring.   Singular value decomposition and low rank approximation of dual complex matrices were studied in \cite{QACLL22}.   We studied eigenvalues and Jordan forms in \cite{QC23}.
	
	{We study dual number symmetric matrices, dual complex Hermitian matrices and dual quaternion Hermitian matrices in a unified frame of dual Hermitian matrices.} Suppose we have a ring, which can be the real field, the complex field, or the quaternion ring.  Then an $n \times n$ dual Hermitian matrix has $n$ dual number eigenvalues.   The trouble {of finding} these $n$ dual number eigenvalues occurs when the standard part of the dual Hermitian matrix has multiple eigenvalues.   {In this case, the characteristic polynomial of that dual Hermitian matrix has infinitely many roots, which are not eigenvalues of that dual Hermitian matrix in general.}   This may make a computational method divergent or slow.
	
	In this paper, we  define {supplement} matrices for a dual Hermitian matrix.  {Supplement} matrices are Hermitian matrices in the original ring.  The standard parts of the eigenvalues of that dual Hermitian matrix are the eigenvalues of the standard part Hermitian matrix in the original ring, while the dual parts of the eigenvalues of that dual Hermitian matrix are the eigenvalues of those  {supplement} matrices.    Hence, by apply any practical method for computing eigenvalues of Hermitian matrices in the original ring, we have a practical method for computing eigenvalues of a dual Hermitian matrix.   {We call this method the {\bf supplement matrix method}.}
	
	Then we study the relative configuration problem in multi-agent formation control.   People need to know if a given desired relative configuration scheme is reasonable such that a feasible solution of configurations of these multi-agents exists.  By combining the eigenvalue problem of dual Hermitian matrices, with the unit gain graph theory, we open a cross-disciplinary approach to solve the relative configuration problem.  Then the supplement matrix method is used for this approach.
	
	In the next section, we review some basic knowledge of dual elements and dual matrices in a ring.  That ring can be the real field, the complex field, or the quaternion ring.   We also review some knowledge about unit dual quaternions there.  This is useful for studying multi-agent formation control.	In Section 3, we  {define the determinants  and characteristic polynomials of dual Hermitian matrices and} show that when the standard part of a dual Hermitian matrix has multiple eigenvalues, the characteristic polynomial of that dual Hermitian matrix may have infinitely many roots, which are not eigenvalues of that dual Hermitian matrix in general.
	We define supplement matrices, construct  the supplement matrix method and
	prove that it does find all eigenvalues of a dual Hermitian matrix in Section 4.     In {Section 5}, we study the relative configuration problem in formation control.   We explore the cross-disciplinary approach to solve the relative configuration problem in Section 6.
	Numerical experiments are reported in Section 7. {Some concluding remarks are made in Section 8.}

	%\section{Dual Elements and Dual Matrices}

	\section{Dual Elements and Dual Matrices}
	
	\subsection{Dual Elements and Unit Dual Quaternions}
	The field of real numbers, the field of complex numbers and the ring of quaternions are denoted  by  $\rset, \cset$ and $\qset$, respectively.   We use $\gset$ to represent them.   Thus, $\gset$ may be $\rset$ or $\cset$ or $\qset$.   We use $\drset, \dcset$ and $\dqset$ to denote the ring of dual numbers, dual complex numbers and dual quaternions, respectively, and use $\dgset$ to represent them in general.   We call an element in $\dgset$ a dual element.   Thus a dual element means a dual number, or a dual complex number, or a dual quaternion, depending upon $\dgset = \drset$, or $\dcset$, or $\dqset$.   Similarly, a Hermitian matrix in $\gset$ means a symmetric matrix, or a complex Hermitian matrix, or a quaternion Hermitian matrix, depending upon $\gset = \rset$, or $\cset$, or $\qset$.
	
	{Our arguments can be generalized to the other rings.   We do not pursue this here.}

	A {\bf dual element} $a = a_s + a_d\epsilon \in \dgset$ has standard part $a_s \in \gset$ and dual part $a_d \in \gset$.    The symbol $\epsilon$ is the infinitesimal unit, satisfying $\epsilon^2 = 0$, and $\epsilon$ is commutative with numbers in $\gset$.  The {\bf conjugate} of $a$ is defined as $a^* = a_s^* + a_d^*\epsilon$, where $a_s^*$ and $a_d^*$ is the conjugates of numbers $a_s$ and $a_d$, respectively.   Note
	that if $a_s$ and $a_d$ are real numbers, then their conjugates are themselves.  Thus, the conjugate of a dual number is also itself.    If $a_s \not = 0$, then we say that $a$ is {\bf appreciable}.   Otherwise, we say that $a$ is {\bf infinitesimal}.
	
	Suppose we have two dual elements $a = a_s + a_d\epsilon$ and $b = b_s + b_d\epsilon$.   Then their sum is $a+b = (a_s+b_s) + (a_d+b_d)\epsilon$, and their product is $ab = a_sb_s + (a_sb_d+a_db_s)\epsilon$.
	In this way, $\dgset$ is a ring.   In particular, $\drset$ and $\dcset$ are two commutative rings, while $\dqset$ is a noncommutative ring.
	
	If both $a_s$ and $a_d$ are real numbers, then $a = a_s + a_d\epsilon$ is called a dual number.  Suppose we have two dual numbers $a = a_s + a_d\epsilon$ and $b = b_s + b_d\epsilon$.   By \cite{QLY22}, if
	$a_s > b_s$, or $a_s = b_s$ and $a_d > b_d$, then we say $a > b$.   Then this defines positive, nonnegative dual numbers, etc.
	%Denote $\mathbb D_+$, $\mathbb D_{++}$ as the set of nonnegative  and positive dual numbers, respectively.
	%In particular,
	%{By \cite{QLY22}, if $a$ is nonnegative and appreciable, then the square root of $a$ is still a nonnegative dual number.   If $a$ is positive and appreciable, we have
	%\begin{equation*}
	%\sqrt{a} = \sqrt{a_s} + {a_d \over 2\sqrt{a_s}}\epsilon.
	%\end{equation*}
	%When $a=0$, we have $\sqrt{a} = 0$.}

	For a dual element $a = a_s + a_d\epsilon \in \dgset$, its magnitude is defined as a nonnegative dual number
	$$|a| := \left\{ \begin{aligned} |a_s| + {(a_sa_d^*+a_da_s^*) \over 2|a_s|}\epsilon, & \ {\rm if}\  a_s \not = 0, \\
		|a_d|\epsilon, &  \ {\rm otherwise}.
	\end{aligned} \right.$$
	%For any  dual   number $a=a_s+a_d\epsilon$ and dual number  $b=b_s+b_d\epsilon$ with $a_s\neq 0$, or $a_s=0$ and $b_s=0$, there is
	%\begin{equation*}
	%\frac{a_s+a_d\epsilon}{b_s+b_d\epsilon} =
	%\left\{
	%\begin{array}{ll}
	%\frac{a_s}{b_s}+\left(  \frac{a_d}{b_s}- \frac{a_s}{b_s} \frac{b_d}{b_s}\right)\epsilon,   & \text{ if } b_s\neq 0, \\
	%\frac{a_d}{b_d} +c\epsilon,  & \ \text{if } a_s= 0, b_s= 0,\\
	%\end{array}
	%\right.
	%\end{equation*}
	%where $c$ is an arbitrary   number.
	
	%Later, when a dual number is   nonnegative or positive, we say it is a nonnegative dual number or a positive dual number respectively.   If we say a number is a nonnegative number or positive number, then that number should be a real number.
	We   use $0$, ${\bf 0}$, and $O$ to denote a zero number, a zero vector, and a zero matrix, respectively.
	%\red{We use $0_d$, ${\bf 0}_d$, and $O_d$ to denote a dual zero number, a dual zero vector, and a dual zero matrix, respectively.}
	
	A dual element $a = a_s + a_d\epsilon \in \dgset$ is called invertible if there exists a dual element $b \in \dgset$ such that $ab = ba =1$.  We can derive that $a$ is invertible if and only if  $a$ is appreciable. In this case, we have
	$$a^{-1} = a_s^{-1} - a_s^{-1}a_d a_s^{-1} \epsilon.$$
	
	%\red{The calculus rules still hold for dual quaternions.   In particular, if $\hat q$ is a function of $t$ and is invertible, then we have
	%\begin{equation} \label{inversederivative}
	%{d \over dt}\hat q^{-1}(t) = - \left(\hat q^{-1}(t)\right)^2 {d \over dt}\hat q(t).
	%\end{equation}}
	
	Let $q = q_s + q_d\epsilon \in \dqset$.  If $|q| = 1$, then $q$ is called a {\bf unit dual quaternion}. A unit dual quaternion $q$ is always invertible and we have $q^{-1} = q^*$.   The 3D motion of a rigid body can be represented by a unit dual quaternion.
	We have
	$$qq^* = (q_s + q_d\epsilon)(q_s^* + q_d^*\epsilon)= q_sq_s^* + (q_sq_d^* + q_d q_s^*)\epsilon = q^*q.$$
	Thus, $q$ is a unit dual quaternion if and only if $q_s$ is a unit quaternion, and
	\begin{equation} \label{udq1}
		q_sq_d^* + q_d q_s^* = q_s^*q_d + q_d^* q_s=0.
	\end{equation}
	Note that a quaternion is called an {\bf imaginary quaternion} if its real part is zero.
	Suppose that there is a rotation $q_s \in \qset$ succeeded by a translation $p^b \in {\mathbb Q}$, where $p^b$ is an imaginary quaternion.   Here, following \cite{WYL12}, we use {the superscript $b$ %and $s$
		to represent} the relation of the rigid body motion with respect to the body frame attached to the rigid body.
	% and the spatial frame that is relative to a fixed coordinate frame.
	Then the whole transformation can be represented using unit dual quaternion $q = q_s + q_d \epsilon$, where $q_d = {1 \over 2}q_sp^b$.   Note that we have
	$$q_sq_d^* + q_d q_s^* = {1 \over 2}\left[q_s(p^b)^*q_s^* + q_sp^bq_s^*\right] = {1 \over 2}q_s\left[(p^b)^*+p^b\right]q_s^* = 0.$$
	Thus, a transformation of a rigid body can be represented by a unit dual quaternion
	\begin{equation} \label{udq}
		q = q_s + {\epsilon \over 2}q_sp^b,
	\end{equation}
	where {$q_s$} is a unit quaternion to represent the rotation, and $p^b$ is  an imaginary quaternion to represent the translation or the position.  On the other hand, every attitude of a rigid body which is free to rotate relative to a fixed frame can be identified by a unique unit quaternion $q$.    Thus, in (\ref{udq}), $q_s$ is the attitude of the rigid body, while {$q_d$} represents the transformation.  A unit dual quaternion {$q$} serves as both a specification of the configuration of a rigid body and a transformation taking the coordinates of a point from one frame to another via rotation and translation.
	In (\ref{udq}), if  {$q$} is the configuration of the rigid body, then {$q_s$} and $p^b$ are the attitude of and position of the rigid body respectively. Denote the set of unit dual quaternions by $\hat {\mathbb U}$.

	%{If a dual number $a = a_s + a_d\epsilon$ is appreciable, then it can also be expressed by its {\bf polar representation} $a = a_s(1 + t\epsilon)$, where $t = {a_d \over a_s}$, or its {\bf exponential representation} $a = a_se^{t\epsilon}$, with $e^{t\epsilon} = 1 + t\epsilon$.}
	
	%\section{Dual Number Vectors}
	
	A {\bf dual vector} is denoted by $\vx = (x_1, \cdots, x_n)^\top \in \dgset^n$.   Its conjugate is $\vx^* = (x_1^*, \cdots, x_n^*)$.  We may denote $\vx = \vx_s + \vx_d\epsilon$, where $\vx_s, \vx_d \in \gset^n$.
	Its $2$-norm is defined as
	\begin{equation}\label{2-norm}
		\|\vx\|_2 = \left\{\begin{array}{ll}
			{\|\vx_s\|_2+{\frac{\vx_s^* \vx_d+\vx_d^* \vx_s}{2\|\vx_s\|_2}}\epsilon,}
			%\sqrt{\sum_{i=1}^n |x_i|^2},
			& \ \mathrm{if} \  \vx_s \not = \0,\\
			\|\vx_d\|_2\epsilon, & \ \mathrm{if} \  \vx_s  = \0.
		\end{array}\right.
	\end{equation}
		{This}  is a dual number.  	For convenience in the numerical experiments, we also  denote
		the $2^R$-norm of a dual   vector $\mathbf x=(x_{i}) \in\hat{\mathbb G}^{n}$  as
		\begin{equation*}
			\|{\mathbf x}\|_{2^R} = \sqrt{\|\mathbf  x_{s}\|_2^2 + \|\mathbf x_{d}\|_2^2},
	\end{equation*}
	which is a real number.

	We say $\vx = \vx_s + \vx_d\epsilon = (x_1, \cdots, x_n)^\top \in \dgset^n$ is a {\bf unit dual vector} if $\|\vx\|_2=1$, or equivalently, $\|\vx_s\|_2=1$ and {$\vx_s^*\vx_d+\vx_d^* \vx_s=0$.
		If $\vx_s \not = \0$, then we say that $\vx$ is appreciable.
		The unit vectors in $\rset^n$ are denoted as $\ve_1, \cdots, \ve_n$.   They are also unit vectors of $\gset^n$ and $\dgset^n$.}

	%Let $\vy = \vy_s + \vy_d\epsilon =\frac{\vx}{\|\vx\|_2}$ be its normalization vector. Then   if $\vx$ is appreciable, we have
	%\begin{equation*}
	%\vy_s = \frac{\vx_s}{\|\vx_s\|_2}, \
	%\vy_d = \frac{\vx_d}{\|\vx_s\|_2}-\frac{\vx_s}{\|\vx_s\|_2} \frac{\vx_s^*\vx_d}{\|\vx_s\|_2^2};
	%\end{equation*}
	%otherwise, if $\vx_s=\0$, we have
	%\begin{equation*}
	%\vy_s = \frac{\vx_d}{\|\vx_d\|_2}, \
	%\vy_d \text{ is any complex vector satisfying }\vy_s^*\vy_d=0.
	%\end{equation*}
	
	Let $\vx, \vy \in \dgset^n$.  If $\vx^* \vy = 0$, then we say that $\vx$ and $\vy$ are orthogonal.  If
	$\vx^{(1)}, \cdots, \vx^{(n)} \in \dgset^n$ and $\left(\vx^{(i)}\right)^*\vx^{(j)} = \delta_{ij}$ for $i, j = 1, \cdots, n$, where $\delta_{ij}$ is the Kronecker symbol, then we say that
	$\{ \vx^{(1)}, \cdots, \vx^{(n)} \}$ is an orthonormal basis of $\dgset^n$.
	Let $\vx^{(1)} = \vx_s^{(1)} + \vx_d^{(1)}\epsilon, \dots, \vx^{(k)} = \vx_s^{(k)} + \vx_d^{(k)}\epsilon \in {\dgset}^n$.  If $\vx_s^{(1)}, \dots, \vx_s^{(k)}$ are linearly independent, then we say that $\vx^{(1)}, \dots, \vx^{(k)}$ are appreciably linearly independent.
	
	%We have the following proposition.
	
	%\begin{Prop}
	%Suppose that $\vx, \vy \in \dgset^n$, and $q \in \dgset$.   Then,
	%\begin{itemize}
	%\item[(i)] $\|\vx\|_2 \ge 0$, and $\|\vx \|_2 = 0$ if and only if $\vx = \0$;
	%\item[(ii)]$\|q\vx\|_2 = |q|\|\vx\|_2$;
	%\item[(iii)]$\|\vx + \vy\|_2 \le \|\vx\|_2 + \|\vy\|_2$.
	%\end{itemize}
	%\end{Prop}

	%This proposition can be proved by definition.  It is a special case of Theorem 6.4 of \cite{QLY22}.  Hence, we omit its proof here.
	
	%Note that if both $\vx$ and $q$ are appreciable, then $q\vx$ is appreciable.
	
	\subsection{Dual Matrices}
	Assume
	that $A = A_s + A_d\epsilon$ and $B = B_s + B_d\epsilon$ are two dual matrices in $\dgset^{n \times n}$, where $n$ is a positive integer, $A_s, A_d, B_s, B_d \in \gset^{n \times n}$.
	If $AB = BA = I$, where
	$I$ is the $n \times n$ identity matrix, then we say that $B$ is the {\bf inverse} of $A$ and denote that $B = A^{-1}$.  %We have the following proposition.
	
	%\begin{Prop} \label{p2.1}
	%Suppose that $A = A_s + A_d\epsilon$ and $B = B_s + B_d\epsilon$ are two dual matrices in $\dgset^{n \times n}$, where $n$ is a positive integer, $A_s, A_d, B_s, B_d \in \gset^{n \times n}$.   Then the following four statements are equivalent.
	
	%(a) $B = A^{-1}$;
	
	%(b) $AB = I$;

	%(c) $A_sB_s = I$ and $A_sB_d+A_dB_s=O$, where $O$ is the $n \times n$ zero matrix;
	
	%(c) $B_s = A_s^{-1}$ and $A_sB_d+A_dB_s=O$.
	
	%(d) $B_s = A_s^{-1}$ and $B_d= -A_s^{-1}A_dA_s^{-1}$.
	%\end{Prop}

	For a dual matrix $A \in \dgset^{n \times n}$, denote its conjugate transpose as $A^*$.  If $A^* = A$, then $A$ is called a {\bf dual Hermitian matrix}.  If $A^* = A^{-1}$, then $A$ is called a {\bf dual unitary matrix}. In particular, if $A$ is a dual Hermitian matrix {in} $\drset^{n \times n}$, or $\dcset^{n \times n}$, or $\dqset^{n \times n}$, respectively, then $A$ is called a dual number symmetric matrix, or a dual complex Hermitian matrix, or a dual quaternion Hermitian matrix, respectively.  If $A$ is a dual  matrix {in} $\drset^{n \times n}$, or $\dcset^{n \times n}$, or $\dqset^{n \times n}$, respectively, then $A$ is called a dual {number  matrix}, or a dual {complex   matrix,} or a dual {quaternion   matrix,} respectively.
	%{For $A = (a_{ij}) \in \dgset^{m\times n}$, its Frobenius norm is defined as
	%$$\|A\|_F = \sqrt{ \sum_{i=1}^m \sum_{j=1}^n |a_{ij}|^2}.$$}
	
	%Apparently,  $A \in \dgset^{n \times n}$ is unitary if and only if its column vectors form an orthonormal basis of $\dgset^n$. %dual Hermitian if $A^\top_\epsilon = A$; dual unitary if $A^\top_\epsilon A = I$; dual complex Hermitian if $A^*_\epsilon = A$; dual complex unitary if $A^*_\epsilon A = I$.
	%Let $k \leq n$. We say that $A \in \dgset^{n\times k}$ is {\bf partially unitary} if its column vectors are unit dual vectors and orthogonal to each other.
	
		{
		The $F$-norm of a dual matrix $A=(a_{ij})\in \hat{\mathbb G}^{m\times n}$ is
		\begin{equation*}\label{equ:dqFnorm}
			\|A\|_F=\left\{\begin{array}{cl}
				\|A_s\|_F+\frac{tr(A_s^*A_d+A_d^*A_s)}{2\|A_s\|_F}\epsilon,   &  \text{ if }A_s\neq O, \\
				\|A_d\|_F\epsilon,   &  \text{ otherwise}.
			\end{array}\right.
		\end{equation*}
	 This is a dual number.
		For convenience in the numerical experiments, we also define  the $F^R$-norm of a dual   matrix $A=(a_{ij}) \in\hat{\mathbb G}^{m\times n}$  as
		\begin{equation*}
			\|{A}\|_{F^R} = \sqrt{\|A_s\|_F^2 + \|A_d\|_F^2},
		\end{equation*}
		which is a real number.
	}

	It is classical that an $n \times n$ real symmetric or complex Hermitian matrix has $n$ real eigenvalues, and this matrix is positive semidefinite (or definite respectively) if and only if all of these $n$ eigenvalues are nonnegative (or positive respectively).   In 1997, Zhang \cite{Zh97} extended this to quaternion Hermitian matrices.  In 2023, Qi and Luo \cite{QL23} further extended this to dual quaternion Hermitian matrices.   This is actually also true for dual symmetric or dual complex Hermitian matrices.  We now state these in our general frame.
	
	Let $A \in \dgset^{n \times n}$ be a dual Hermitian matrix and $\vx \in \dgset^n$.  Then $\vx^*A\vx$ is a dual number if $\dgset$ is either $\drset$, or $\dcset$ or $\dqset$. Thus, by \cite{QLY22}, we may distinguish that $\vx^*A\vx$ is nonnegative, or positive, or not. If for all $\vx \in \dgset^n$,  $\vx^*A\vx$ is nonnegative, then we say that $A$ is positive semidefinite.   If for all $\vx \in \dgset^n$ and appreciable,  $\vx^*A\vx$ is positive, then we say that $A$ is positive definite.
	
	Let $A \in \dgset^{n \times n}$, $\vx \in \dgset^n$ {be} appreciable, and $\lambda \in \dgset$.
	If
	\begin{equation} \label{en1}
		A\vx = \vx\lambda,
	\end{equation}
	then $\lambda$ is called a right eigenvalue of $A$, with $\vx$ {as its corresponding} right eigenvector.  If
	\begin{equation} \label{en2}
		A\vx = \lambda\vx,
	\end{equation}
	where $\vx$ is appreciable, i.e., $\vx_s \not = \0$, then $\lambda$ is called {a left} eigenvalue of $A$, with {a left} eigenvector $\vx$.
	
	If $\dgset$ is $\drset$ or $\dcset$, then the multiplication is commutative.    In these two cases, it is not {necessary} to distinguish right and left eigenvalues.  We just call them eigenvalues \cite{QC23}.   It was proved in \cite{QL23} that all the right eigenvalues of a dual quaternion Hermitian matrix $A$ are dual numbers.   As dual numbers are commutative with dual quaternions, they are also left eigenvalues.  Thus, we may simply call them {\bf eigenvalues} of $A$.  Note that $A$ may still have other left eigenvalues, which are not dual numbers.  See an example of a quaternion matrix in \cite{Zh97}.
	
	The following theorem was proved in \cite{QL23} for dual quaternion Hermitian matrices.  It is also true for dual symmetric matrices and dual complex Hermitian matrices by similar arguments \cite{QC23}.
	
	\begin{Thm}   Suppose that $A \in \dgset^{n \times n}$ is a dual symmetric matrix, or a dual complex Hermitian matrix, or a dual quaternion Hermitian matrix, then it has exactly $n$ dual number eigenvalues.   It is positive semidefinite  or definite if and only if these $n$ eigenvalues are nonnegative {or positive}.
	\end{Thm}
	
	Write $A = A_s + A_d\epsilon$, $\lambda = \lambda_s + \lambda_d \epsilon$ and $\vx = \vx_s + \vx_d \epsilon$.  Then
	(\ref{en2}) is equivalent to
	\begin{equation} \label{en4}
		A_s\vx_s = \lambda_s\vx_s,
	\end{equation}
	with $\vx_s \not = \0$, i.e., $\lambda_s$ is an eigenvalue of $A_s$ with an eigenvector $\vx_s$, and
	\begin{equation} \label{en5}
		(A_s-\lambda_sI)\vx_d - \lambda_d\vx_s = -A_d\vx_s.
	\end{equation}

	%\red{If $\lambda_r>0$, then the rank of $A$ is $n$. Otherwise, if $\lambda_r=0$, then we call $\sum_{i=1}^{r-1} k_i$ the appreciable rank of $A$, and  $\sum_{i=1}^{r-1} k_i + t_r$ the rank of $A$. Here, $t_r$ is the number of nonzero elements in $\{\lambda_{r,1},\dots,\lambda_{r,k_r}\}$.}

	%For more knowledge of dual complex matrices, please see \cite{QC23, QLY22, WDW24}.
	\bigskip
	
	\section{The Characteristic Polynomial of a Dual Hermitian Matrix}

	Suppose that $A  \in \dgset^{n \times n}$ is an $n \times n$ dual Hermitian matrix.  Here, $\dgset$ is either $\drset$, or $\dcset$, or $\dqset$.
	Then $A$ has $n$ dual number eigenvalues $\lambda_1, \dots, \lambda_n$.   We may define its determinant as
	\begin{equation}\label{def:det}
		det(A) = \lambda_1\lambda_2\dots \lambda_n.
	\end{equation}
	Furthermore, we may define the characteristic polynomial as
	$${\phi(\lambda) = det(\lambda I - A),}$$ where
	$I$ is the $n \times n$ identity matrix.
	This shows that $\phi$ is a dual number polynomial.   Note that the factorization form of a dual number polynomial is not unique.   For example, polynomial $\lambda^2 = (\lambda + a\epsilon)(\lambda-a\epsilon)$ for any real number $a$.
	This leads  the  characteristic polynomial {of the  dual Hermitian matrix} being more complicated than that of the   Hermitian matrix.
	
	\begin{Thm} \label{charact}
		Suppose that $A = A_s + A_d\epsilon \in \dgset^{n \times n}$ is an $n \times n$ dual Hermitian matrix.  Here, $\dgset$ is either $\drset$, or $\dcset$, or $\dqset$.
		Let its $n$ dual number eigenvalues be $\lambda_1, \dots, \lambda_n$.
		Then its characteristic polynomial has the form
		\begin{equation} \label{char}
			\phi(\lambda) \equiv (\lambda - \lambda_1)\dots (\lambda - \lambda_n).
		\end{equation}
		Furthermore, a dual number $\lambda = \lambda_s + \lambda_d\epsilon$ is a root of $\phi$, either if  $\lambda$ is an eigenvalue of $A$, or if $\lambda_s$ is a multiple eigenvalue of $A_s$.
	\end{Thm}
	\begin{proof}
		First, we see that $\lambda I -A$ is an $n \times n$ dual Hermitian matrix too.   Since $A$ has $n$ dual number eigenvalues $\lambda_1, \dots, \lambda_n$, by the definition of eigenvalues of dual Hermitian matrices, we see that $\lambda I -A$ has $n$ dual number eigenvalues $\lambda - \lambda_1, \dots, \lambda - \lambda_n$. By the definition of determinants and characteristic polynomials, we have (\ref{char}).
		
		Clearly, any eigenvalue of $A$ is a root of $\phi$.  By (\ref{en4}), the standard part of an eigenvalue of $A$ is an eigenvalue of $A_s$.  Let $\lambda = \lambda_s + \lambda_d\epsilon$ be a dual number.  If $\lambda_s$ is not an eigenvalue of $A_s$, then $\lambda - \lambda_i$ is an appreciable dual number for $i = 1, \cdots, n$.  By (\ref{char}), this implies $\phi(\lambda) \not = 0$, i.e., $\lambda$ is not a root of $\phi$.   If $\lambda_s$ is a single eigenvalue of $A_s$, then $\lambda_s$ is equal to the standard part of $\lambda_i$ for one $i$.  Then $\lambda - \lambda_j$ is an appreciable dual number for $j \not = i$.  By (\ref{char}), $\phi(\lambda) = 0$ if and only if $\lambda = \lambda_i$.   Finally, assume that $\lambda_s$ is a multiple eigenvalue of $A_s$.   Without loss of generality, we may assume that both the standard parts of $\lambda_1$ and $\lambda_2$ are $\lambda_s$.  Then $(\lambda - \lambda_1)(\lambda - \lambda_2) = 0$.  This implies that $\phi(\lambda) = 0$.   This completes the proof.
	\end{proof}
	
	%This shows that $\phi$ is a dual number polynomial.   Note that the factorization form of a dual number polynomial is not unique.   For example, polynomial $\lambda^2 = (\lambda + a\epsilon)(\lambda-a\epsilon)$ for any real number $a$.
	
	Theorem \ref{charact} reveals that when $A_s$ has multiple eigenvalues,  the characteristic polynomial of $A$ may have infinitely many roots, which are not eigenvalues of $A$ in general.  Therefore, it will be difficult to find eigenvalues of $A$ {by the characteristic polynomial} if we treat the standard parts and dual parts of the eigenvalues of $A$ together.  %A better
	Another way is to find eigenvalues and eigenvectors of $A_s$ by a classical method first.  Then, the questions are, if $\lambda_s$ is a single eigenvalue of $A_s$, can we give a formula of $\lambda_d$ such that $\lambda_s + \lambda_d\epsilon$ is an eigenvalue of $A$?  And if $\lambda_s$ is a $k$-multiple eigenvalue of $A_s$, how can we find the dual parts of corresponding {$k$ eigenvalues} of $A$?   In the next section, we will answer these two questions.

	\bigskip
	
	\section{A Practical Method for Computing Eigenvalues of a Dual Hermitian Matrix}
	
	In this section, we define {supplement matrices} for a dual Hermitian matrix $A \in \dgset^{n \times n}$ to enable the calculation of the eigenvalues of $A$.   Assume that {$\gset$ is} either $\rset$ or $\cset$ or $\qset$.
	
	\begin{Thm}\label{thm:practical}
		Suppose that $A = A_s + A_d\epsilon \in \dgset^{n \times n}$ is a dual Hermitian matrix, and {the} real number $\lambda_s $ is a $k$-multiple eigenvalue of the Hermitian matrix $A_s \in \gset^{n \times n}$.
		
		If $k = 1$, i.e., $\lambda_s$ is a single eigenvalue of $A_s$, let $\vx_s$ be a unit eigenvector of $A_s$, associated with $\lambda_s$.  Then
		$\lambda = \lambda_s + \lambda_d\epsilon$ is a single eigenvalue of $A$, where $\lambda_d = \vx_s^*A_s\vx_s$,
		with an eigenvector $\vx = \vx_s + \vx_d\epsilon$, where $\vx_d$ is a solution of
		\begin{equation} \label{en6}
			(\lambda_sI-A_s)\vx_d = (A_d -\lambda_dI)\vx_s.
		\end{equation}
		
		If $k > 1$, let $\vv_1, \dots, \vv_k$ be $k$ orthonormal eigenvectors of $A_s$, associated with $\lambda_s$.  Let $W = (\vv_1 \dots, \vv_k)$.   Then $W$ is an $n \times k$ partially unitary matrix, and $W^*A_dW \in \gset^{k \times k}$ is a Hermitian matrix.   Let $\lambda_{d1}, \dots, \lambda_{dk}$ be the $k$ eigenvalues of $W^*A_dW$, with corresponding eigenvectors $\vy_1, \dots, \vy_k$.  Then $\lambda_i = \lambda_s + \lambda_{di}{\epsilon}$ for $i = 1, \dots, k$ are eigenvalues of $A$, with eigenvectors $\vx_i = \vx_{si}+ \vx_{di}\epsilon$ for $i = 1, \dots, k$, where $\vx_{si} = W\vy_i$, and $\vx_{di}$ is a solution of
		\begin{equation} \label{en7}
			(\lambda_sI-A_s)\vx_{di} = (A_d -\lambda_{di}I)\vx_{si},
		\end{equation}
		for $i = 1, \dots, k$.
	\end{Thm}
	\begin{proof}
		Suppose that $k = 1$, i.e., $\lambda_s$ is a single eigenvalue of $A_s$.   Let $\vx_s$ be a unit eigenvector of $A_s$, associated with $\lambda_s$.  Multiply (\ref{en5}) by $\vx^*_s$ from {the} left.  Then we have $\lambda_d = \vx_s^*A_d\vx_s$.  From (\ref{en5}), we have (\ref{en6}).
		
		Suppose that $k > 1$.  Let $\vv_1, \dots, \vv_k$ be $k$ orthonormal eigenvectors of $A_s$, associated with $\lambda_s$.  Let $W = (\vv_1 \dots, \vv_k)$.   Then $W$ is an $n \times k$ partially unitary matrix, and $W^*A_dW \in \gset^{k \times k}$ is a Hermitian matrix.   Let $\lambda_{d1}, \dots, \lambda_{dk}$ be the $k$ eigenvalues of $W^*A_dW$, with corresponding eigenvectors $\vy_1, \dots, \vy_k$.  Then we have $W^*A_dW\vy_i = \lambda_{di}\vy_i$, for $i = 1,\dots, k$.  Let $\vx_{si} = W\vy_i$, for $i = 1, \dots, k$.
		Then we have $A_s\vx_{si} = \lambda_s\vx_{si}$, i.e., $\vx_{si}$ is an eigenvector of $A_s$, associated with eigenvalue $\lambda_s$, for $i = 1, \dots, k$.   Now (\ref{en5}) has the form
		$$(A_s-\lambda_sI)\vx_{di} - \lambda_d\vx_{si} = -A_d{\vx_{si}},$$
		for $i = 1, \dots, k$.   Multiply the above equality by $\vx_{si}^*$ from {the} left.  Then {we have $\lambda_{d}=\vx_{si}^*A_d\vx_{si}=\vy_i^*W^*A_dW\vy_i =\lambda_{di}$, and} $\lambda_i = \lambda_s + \lambda_{di}{\epsilon}$ for $i = 1, \dots, k$ are eigenvalues of $A$, with eigenvectors $\vx_i = \vx_{si}+ \vx_{di}\epsilon$, where $\vx_{di}$ satisfies (\ref{en7}),
		for $i = 1, \dots, k$.
		
		This completes the proof.
	\end{proof}
	
	We call the Hermitian matrix $W^*A_dW \in \gset^{k \times k}$, the {\bf supplement matrix} of the dual Hermitian matrix $A$, corresponding to the $k$-multiple eigenvalue $\lambda_s$ of $A_s$, the standard part of $A$.
	
	Then we have a practical method for computing eigenvalues of the dual Hermitian matrix $A$.
	
	\medskip
	
	\begin{Alg} {\bf The Supplement Matrix Method {(SMM)}} \label{alg:SMM}
		Suppose that $A = A_s + A_d\epsilon \in \dgset^{n \times n}$ is a dual Hermitian matrix.
		
		Step 1.  Use a practical method to find $n$ real eigenvalues of the Hermitian matrix $A_s \in \gset^{n \times n}$, with a set of orthonormal eigenvectors.
		
		Step 2.  Assume that $\lambda_s$ is a single eigenvalue of $A_s$ with a unit vector $\vx_s$.  Then
		$\lambda = \lambda_s + \lambda_d\epsilon$ is a single eigenvalue of $A$, where $\lambda_d ={ \vx_s^*A_d\vx_s}$,
		with an eigenvector $\vx = \vx_s + \vx_d\epsilon$, where $\vx_d$ is a solution of (\ref{en6}).
		
		Step 3. Assume that $\lambda_s$ is a $k$-multiple eigenvalue of $A_s$ with $k$ orthonormal eigenvectors
		${\vv_1}, \dots, \vv_k$, for $k > 1$.   Let $W = (\vv_1 \dots, \vv_k)$ and $W^*A_dW \in \gset^{k \times k}$.   Use a practical method to find $k$ eigenvalues of $W^*A_dW$, as $\lambda_{d1}, \dots, \lambda_{dk}$, with corresponding eigenvectors $\vy_1, \dots, \vy_k$.  Then $\lambda_i = \lambda_s + \lambda_{di}{\epsilon}$ for $i = 1, \dots, k$ are eigenvalues of $A$, with eigenvectors $\vx_i = \vx_{si}+ \vx_{di}\epsilon$ for $i = 1, \dots, k$, where $\vx_{si} = W\vy_i$, and $\vx_{di}$ is a solution of (\ref{en7}) for $i = 1, \dots, k$.
		
		Step 4.  Apply Step 2 to all single eigenvalues of $A_s$, and Step 3 to all multiple eigenvalues of $A_s$.
	\end{Alg}

	The linear systems  (\ref{en6}) and (\ref{en7}) may be computed by several methods.
	Suppose the full eigenvalue decomposition of $A_s$ is known as $A_s=U_s\Sigma_sU_s^*$. Then we have
	\begin{eqnarray*}
		\vx_{di} &=&(\lambda_sI-A_s)^+ (A_d -\lambda_dI)\vx_{si} \\
		&=&U_s(\Sigma_s-\lambda_sI)^+U_s^*(A_d -\lambda_dI)\vx_{si}\\
		&=&U_s(\Sigma_s-\lambda_sI)^+U_s^*A_d\vx_{si},
	\end{eqnarray*}
	for $i=1,\dots,k$.
	Otherwise, if the full eigenvalue decomposition of $A_s$ is  {not} known in advance, then we can solve the linear systems  (\ref{en6}) and (\ref{en7})  directly.
	The both systems are consistent and solvable.
	It should be noted that a quaternion linear system may not be easy to solve since the quaternion numbers {are} not communicative.
	One {practical} way is to reformulate the quaternion linear system as a real linear system.
	{Furthermore, both two systems are ill-conditioned since the the coefficient martices are singular, which may lead the numerical  algorithms being unstable. One  possible way is to add extra conditions  $\vx_i^*\vx_j=\delta_{ij}$ for $i,j=1,\dots,k$.}

	Similarly, Algorithm \ref{alg:SMM} can be extended to compute several extreme eigenpairs   or  a few eigenpairs of a dual Hermitian matrix.
	One difficulty here is that if we do not know the multiplicity of the eigenvalue in advance,  we may not obtain the eigenvectors exactly. One way to solve this is to compute {several} extra eigenvalues in the standard part.

	\bigskip
	
	\section{The Relative Configuration Problem in Formation Control}

	Consider the formation control problem of $n$ rigid bodies.   These $n$ rigid bodies can be autonomous mobile robots, or unmanned aerial vehicles (UAVs), or autonomous underwater vehicles (AUVs), or small satellites.
	Then these $n$ rigid bodies can be described by a graph $G=(V,E)$ with $n$ vertices and $m$ edges.  For two rigid bodies $i$ and $j$ in $V$, if rigid body $i$ can sense rigid body $j$, then edge $(i, j) \in E$.    As studied in \cite{QWL23}, we assume that any pair of these rigid bodies are mutual visual, i.e., rigid body $i$ can sense rigid body $j$ if and only if rigid body $j$ can sense rigid body $i$.   Furthermore, we assume that $G$ is connected in the sense that for any node pair $i$ and $j$, either $(i, j) \in E$, or there is a path connecting $i$ and $j$ in $G$, i.e., there are nodes $i_i, \dots, i_k \in V$ such that $(i, i_1), \dots, (i_k, j) \in E$.
	
	Suppose that for each $(i, j) \in E$, we have a desired relative configuration  {from rigid body $i$ to $j$ as} $q_{d_{ij}} \in \hat{\mathbb U}$.
	%		 By \cite{QC24}, the desired relative configuration $\left\{q_{d_{ij}} : (i, j) \in E \right\}$ is {\bf reasonable} if and only if there is a desired formation $q_d\in\hat{\mathbb U}^{n\times 1}$, which satisfies
	%		\begin{equation} \label{desrel}
	%		  q_{d_{ij}} =  q_{d_i}^* q_{d_j},
	%		\end{equation}
	%		for all $(i, j) \in E$.
	%		Then in Theorem 4.1 \cite{QC24}, Qi and Cui showed that the desired relative configuration is reasonable  if and only if for all $(i, j) \in E$,  there is
	%			\begin{equation} \label{Hermitian}
	%			  q_{d_{ji}} =  q_{d_{ij}}^*;
	%			\end{equation}
	%		and for any cycle $\left\{ j_1, \dots, j_k \right\}$, with $j_{k+1} = j_1$, of $G$, it holds that
	%			\begin{equation} \label{cycle}
	%				\prod_{i=1}^k   q_{d_{j_ij_{i+1}}} = 1.
	%			\end{equation}
	We say that the desired relative configurations $\left\{ q_{d_{ij}} : (i, j) \in E \right\}$ is {\bf reasonable} if and only if there is a desired formation $\vdq q_d\in\hat{\mathbb U}^{n\times 1}$, which satisfies
	\begin{equation} \label{desrel}
		q_{d_{ij}} = q_{d_i}^*q_{d_j},
	\end{equation}
	for all $(i, j) \in E$.
	
	Thus, a meaningful application problem in formation control is to verify  {whether} a given desired relative configuration scheme is reasonable or not.   See \cite{LWHF14, QWL23, WYL12} for more study on formation control.
	
	\begin{Thm}  \label{Reasonable} Suppose that we have an undirected graph $G = (V, E)$, which is bidirectional and connected.
		The desired relative configurations $\left\{ q_{d_{ij}} : (i, j) \in E \right\}$ is {\bf reasonable} if and only if
		\begin{itemize}
			
			\item for all $(i, j) \in E$, we have
			\begin{equation} \label{Hermitian}
				q_{d_{ji}} = q_{d_{ij}}^*;
			\end{equation}
			
			\item for any cycle $\left\{ j_1, \dots, j_k \right\}$, with $j_{k+1} = j_1$, of $G$, we have
			\begin{equation} \label{cycle}
				\prod_{i=1}^k q_{d_{j_ij_{i+1}}} = 1.
			\end{equation}
		\end{itemize}
		
		Furthermore, let {${\vdq q}_0$} be a desired formation satisfying (\ref{desrel}).  Then the set of all the desired formations satisfying (\ref{desrel}) is
		\begin{equation} \label{alldesiredformation}
			\left\{ c \vdq q_0 : c \in \hat{\mathbb U} \right\}.
		\end{equation}
	\end{Thm}
	\begin{proof}   Suppose that the desired relative configurations $\left\{ q_{d_{ij}} : (i, j) \in E \right\}$ is reasonable.   Then there is a desired formation $\vdq q_d\in\hat{\mathbb U}^{n\times 1}$, which satisfies (\ref{desrel}) for all $(i, j) \in E$.  Then for all $(i, j) \in E$, we have $q_{d_{ji}} = q_{d_j}^*q_{d_i} = q_{d_{ij}}^*$, which implies (\ref{Hermitian}).  On the cycle $\left\{ j_1, \dots, j_k \right\}$, we have $q_{d_{j_{i+1}}} = q_{d_{j_i}}q_{d_{{j_i}{j_{i+1}}}}$
		for $i = 1, \dots, k$.  This implies (\ref{cycle}).
		
		On the other hand, suppose that the desired relative configurations {$\left\{ q_{d_{ij}} \right\}$} satisfy (\ref{Hermitian}) for all $(i, j) \in E$, and (\ref{cycle}) for all cycles in $G$.   We wish to show that there is a desired formation $\vdq q_d\in\hat{\mathbb U}^{n\times 1}$, which satisfies
		(\ref{desrel}) for all $(i, j) \in E$.  {We count two edges $(i, j)$ and $(j , i)$ as one bidirectional edge pair.   Thus, $G$ has $m_0 = {m \over 2}$ bidirectional edge pairs.   By graph theory, since $G$ is bidirectional and connected, we have $m_0 \ge n - 1$.   Let $M = m_0 - n +1$.  We now show by induction on $M$, that there is a desired formation $\vdq q_d\in\hat{\mathbb U}^{n\times 1}$, which satisfies (\ref{desrel}) for all $(i, j) \in E$.} 		
		
		(i)  {We first show this for $M = 0, i.e., m_0 = n-1$.}   Then $G$ is a bidirectional tree.
		{Without loss of generality, we pick node $1$ as the root of this tree.}
		Let $q_{d_1} = 1$.
		Each node in $\{2,\dots,n\}$ has one unique father node.
		Denote $V_{(1)}$ as the set of children nodes of the root.
		{Let} $q_{d_i} = q_{d_1}q_{d_{1i}}=q_{d_{1i}}$ for all $i\in V_{(1)}$.
		Similarly, denote $V_{(j)}$ as the set of children nodes of  $V_{{(j-1)}}$.
		Then $q_{d_i} = q_{d_{f_i}}q_{d_{f_ii}}$ for all $i\in V_{(j)}$,  {where $f_i$ is the} father node  {of node $i$}.  		
		We repeat this process until $V=\cup_j V_{(j)}$.  By \eqref{Hermitian},  equation \eqref{desrel} is satisfied {for all $(i, j) \in E$}.

		(ii) We now assume that this claim is true for $M = M_0$, and prove that it is true for $M = M_0 +1$.   {Then $M > m_0 - n + 1$.  This implies that} there is at least one cycle  $\left\{ j_1, \dots, j_k \right\}$, with $j_{k+1} = j_1$, of $G$, such that the $k$ nodes $j_1, \dots, j_k$ are all distinct and $k \ge 3$. Delete the bidirectional edge pair $(j_1, j_k)$ and $(j_k, j_1)$.   We get an  {undirected} graph $G_0 = (V, E_0)$.  Then $G_0$ is still bidirectional and {connected}, and the desired relative configurations $\left\{ q_{d_{ij}} : (i, j) \in E_0 \right\}$ satisfy (\ref{Hermitian}) for all $(i, j) \in E_0$, and (\ref{cycle}) for all cycles in $G_0$.  By our induction assumption, there is a desired formation $\vdq q_d\in\hat{\mathbb U}^{n\times 1}$ such that (\ref{desrel}) holds
		for all $(i, j) \in E_0$.  Now,  by (\ref{Hermitian}) for $i = j_1$ and $j = j_k$, and (\ref{cycle}) for this cycle $\left\{ j_1, \dots, j_k \right\}$,  (\ref{desrel}) also holds for $(j_1, j_k)$ and $(j_k, j_1)$.  This proves this claim for $M = M_0+1$ in this case.

		Suppose that  ${\vdq q}_0$ is a desired formation satisfying (\ref{desrel}), and $c \in \hat{\mathbb U}$.  Then $c {\vdq q}_0$ also satisfies (\ref{desrel}).   We may prove the other side of the last claim by induction as above.
		
		The proof is {completed}.
	\end{proof}
	The first condition in \eqref{Hermitian} is  easy to verify. However, the second condition in \eqref{cycle} is relatively complicated since the number of cycles may increase exponentially with the number of nodes.
	Fortunately, we can solve this problem with the help of eigenvalues of dual Hermitian matrices, as shown in the next section.
	
	{\section{A Cross-Disciplinary Approach for the Relative Configuration Problem}}

	{{In this section, we show that}  condition \eqref{cycle} is equivalent to  the {balance} of the corresponding gain graph $\Phi$  \cite{CLQW24}.
		Here, $\Phi=(G,\dgset, \varphi)$ is a   dual  unit gain graph with $n$ vertices,     $G=(V,E)$ is the underlying graph,    {$n=|V|$, $m=|E|$,} $\dgset$ is the  gain group, and $\varphi: {E}(\Phi)\rightarrow {\dgset}$ is the gain function such that $\varphi(e_{ij}) = \varphi^{-1}(e_{ji})$.
		{For instance, if $\dgset$ is the set of unit dual complex numbers, then $\Phi$ is a  dual complex unit gain graph.
			If $\dgset$ is the set of unit dual quaternion numbers, then $\Phi$ is a  dual quaternion unit gain graph.}
		The adjacency and Laplacian matrices of $\Phi$ {are} defined
		via the gain function  {$\varphi(e)$} as follows \cite{CLQW24},
		\begin{equation}\label{mat:Adj_phi}
			a_{ij}{(\Phi)} = \left\{
			\begin{array}{cl}
				\varphi(e_{ij}),  &  \text{if } e_{ij}\in {E(\Phi)},\\
				0, & \text{otherwise},
			\end{array}
			\right.  \ \text{ and } L(\Phi) =   D- A(\Phi).
		\end{equation}
		Here,  {$\varphi(e_{ij}) \in {\dgset}$, $\varphi(e_{ij})=\varphi(e_{ji})^{-1}$,}  and $D$ the degree matrix of the corresponding  underlying graph 		$G$.
		{If $|\varphi(e_{ij})|=1$ for all  $e_{ij}\in {E(\Phi)}$, then}
		$A(\Phi)$ and {$L(\Phi)$} are Hermitian matrices in $\dgset^{n\times n}$.

		We define $\varphi(e_{ij}) = q_{d_{ij}}$ for the formation control problem here.
		Then we may apply spectral graph {theory.}
		Let $A$ and $L$ be the adjacency and Laplacian matrices of $\Phi$ defined by \eqref{mat:Adj_phi}, respectively.
		Recently, \cite{CLQW24} showed that  if $\Phi$ is balanced, then $A$ and $L$  are similar with the adjacency and Laplacian matrices of the underlying graph $G$, respectively, and
		its spectrum
		$\sigma_A(\Phi)$ consists of $n$ real numbers, while its Laplacian spectrum $\sigma_L(\Phi)$ consists of one zero and $n-1$ positive numbers.
		Furthermore,  there is $\sigma_A(\Phi)=\sigma_A(G)$ and $\sigma_L(\Phi)=\sigma_L(G)$.
		
		Based on the above facts,  we may verify the reasonableness of a desired relative configuration  by computing all eigenvalues of the adjacency or Laplacian matrices of $\Phi$. If  $\sigma_A(\Phi)=\sigma_A(G)$ or $\sigma_L(\Phi)=\sigma_L(G)$, then the gain graph is balanced and the desired relative configuration is reasonable. Otherwise, the desired relative configuration   is not reasonable. The method is applicable for low and medium dimensional problems.
		
		In the following, we propose verifying the reasonableness of the desired relative configuration by computing the smallest eigenvalue of the Laplacian matrix.
		\begin{Thm}\label{thm:balanced}
			Let $\Phi=(G,\dgset, \varphi)$ be a gain graph with $n=|V|$, $m=|E|$. Suppose that $\varphi(e_{ij})=1$ for all $(ij)\in E$ and  $G$ has $t$ subgraphs $G_i=(V_i,E_i)$, $n=n_1+\dots+n_t$, $V_i=\{n_1+\dots+n_{i-1}+1, \dots,n_1+\dots+n_i\}$, and {$n_0=0$.} Let   $L$ be the Laplacian {matrix} of $\Phi$ defined by \eqref{mat:Adj_phi}.  Then $\Phi$ is balanced if and only if  the following conditions hold simultaneously:
			
			(i) $L$  has $t$ zero eigenvalues that are the smallest eigenvalue of $L$;
			
			(ii)  their  eigenvectors $\vx_i\in\dgset^n$ satisfies {$|x_i(j)|=\frac{\sqrt{n_i}}{n_i}$} for $j=n_1+\dots+n_{i-1}+1, \dots,n_1+\dots+n_i$,   $|x_i(j)|=0$ otherwise;
			
			(iii) $Y^*LY$ is equal to the Laplacian matrix of the underlying graph $G$, where $Y=\text{diag}(\vy)\in\dgset^{n\times n}$, $\vx_a=\vx_1+\dots+\vx_t$, and $y(i) = \frac{x_a(i)}{|x_a(i)|}$.
		\end{Thm}
		\begin{proof}
			The results {follow} directly from the fact that $\Phi$ is balanced if and only if  there exists a diagonal matrix $Y$ such that $Y^*LY$ is equal to the Laplacian matrix of the underlying graph $G$ and the spectral theory of the  Laplacian matrix of a graph $G$.
	\end{proof}}
	
	The gain graph theory is a well-developed area in spectral graph theory.   A gain graph assigns an element of a mathematical group to each of its edges, and if a group element is assigned to an edge, then the inverse of that group element is always assigned to the inverse edge of that edge.    {If such a mathematical group} consists of unit numbers of a number system,  {then} the gain graph is called a unit gain graph.   The real unit gain graph is called a signed graph \cite{HLP03}.  There are {studies} on complex unit gain graphs and quaternion unit gain graphs \cite{CLQW24}.  What we used here for the formation control problem are dual quaternion unit gain graphs, which are not in the literature yet.  By exploring the eigenvalue problem of dual Hermitian matrices, and its link with the unit gain graph theory, we opened a cross-disciplinary approach to solve the relative configuration problem {in formation} control.

	\bigskip
	
	\section{Numerical Experiments}
	
	%	
	%	Let $\Phi=(G,\dgset, \varphi)$ be a gain graph.
	%	Here, $G=(V,E)$ is the underlying graph,    {$n=|V|$, $m=|E|$,} $\dgset$ is the  gain group, and $\varphi: {E}(\Phi)\rightarrow {\dgset}$ is the gain function such that $\varphi(e_{ij}) = \varphi^{-1}(e_{ji})$.
	%	{For instance, if $\dgset$ is the set of unit dual complex numbers, then $\Phi$ is a  dual complex unit gain graph.
	%		If $\dgset$ is the set of unit dual quaternion numbers, then $\Phi$ is a  dual quaternion unit gain graph.}
	%	The adjacency and Laplacian matrices of $\Phi$ {are} defined
	%	via the gain function  {$\varphi(e)$} as follows \cite{CLQW24},
	%	\begin{equation}\label{mat:Adj_phi}
	%		a_{ij}{(\Phi)} = \left\{
	%		\begin{array}{cl}
	%			\varphi(e_{ij}),  &  \text{if } e_{ij}\in {E(\Phi)},\\
	%			0, & \text{otherwise},
	%		\end{array}
	%		\right.  \ \text{ and } L(\Phi) =   D- A(\Phi).
	%	\end{equation}
	%	Here,  {$\varphi(e_{ij}) \in {\dgset}$, $\varphi(e_{ij})=\varphi(e_{ji})^{-1}$,}  and $D$ the degree matrix of the corresponding  underlying graph 		$G$.
	%	{If $|\varphi(e_{ij})|=1$ for all  $e_{ij}\in {E(\Phi)}$, then}
	%	$A(\Phi)$ and {$L(\Phi)$} are Hermitian matrices in $\dgset^{n\times n}$.
	%	
	
	We begin with a toy example of computing all eigenpairs of the adjacency matrix of a dual complex unit gain graph.
	
	\begin{example}
		Consider  a dual complex unit gain  cycle  {$\Phi_A$}  in  Figure 1~(a).
		The adjacency   matrix is given as follows.
		\[A = {\begin{bmatrix}
				0 & 1+i\epsilon & 1-2i\epsilon \\
				1-i\epsilon & 0 &  1 -i\epsilon  \\
				1+2i\epsilon &  1 +i\epsilon  & 0
			\end{bmatrix}.}	\]
		In this example, the eigenvalues of $A_s$ are {$\lambda_{s1}=2$, $\lambda_{s2}=\lambda_{s3}=-1$,}
		and their corresponding  eigenvectors are
		\[  {\vv_1}= \begin{bmatrix}			0.5774 \\   0.5774 \\  0.5774\\
		\end{bmatrix},\quad
		{\vv_2}= \begin{bmatrix}
			-0.7152  \\  0.0166 \\  0.6987\\
		\end{bmatrix},\quad
		{\vv_3}= \begin{bmatrix}
			0.3938 \\  -0.8163   \\ 0.4225\\
		\end{bmatrix}.\]
		The first eigenvalue is single. Thus $\vx_{s1}={\vv_1}$, and
		\[\lambda_{d1} =\vx_{s1}^*{A_d}\vx_{s1}=0,\quad \vx_{d1}= \begin{bmatrix}
			0.1925i  \\ 0.3849i  \\ - 0.5774i \\
		\end{bmatrix}.\]

		The standard parts of  the  second and the third eigenvalues are the same.
		Let $W=[{\vv_2, \vv_3}]$. Then the {supplement} matrix is
		\[ W^*A_dW=\begin{bmatrix}
			0 & 1.1547i \\
			- 1.1547i & 0 \\
		\end{bmatrix},\]
		and its eigenparis are
		\[\lambda_{d2}= 1.1547, \quad \lambda_{d3}=-1.1547,\quad \vy_1=\begin{bmatrix}
			0.7071 \\  - 0.7071i
		\end{bmatrix},\quad \vy_2=\begin{bmatrix}
			0.7071 \\   0.7071i
		\end{bmatrix},\]
		respectively.
		Furthermore, there is
		\[\vx_{s2}=W\vy_1=\begin{bmatrix}
			-0.5058 - 0.2785i  \\ 0.0117 + 0.5772i  \\ 0.4940 - 0.2988i\\
		\end{bmatrix},\quad
		\vx_{s3}=W\vy_2=\begin{bmatrix}
			-0.5058 + 0.2785i  \\ 0.0117 - 0.5772i \\  0.4940 + 0.2988i\\
		\end{bmatrix}. \]
		At last, {we solve \eqref{en7} and derive that}
		\[\vx_{d2}=\begin{bmatrix}
			0.1969 - 0.2183i \\  0.1969 - 0.2183i \\  0.1969 - 0.2183i\\
		\end{bmatrix},\quad
		{\vx_{d3}}=\begin{bmatrix}
			-0.1969 - 0.2183i \\ -0.1969 - 0.2183i \\ -0.1969 - 0.2183i\\
		\end{bmatrix}. \]
	\end{example}

	\begin{figure}[t]\label{ex:dcugg}
		\begin{subfigure}[b]{0.45\textwidth}
			\centering
			\includegraphics[width=0.65\linewidth]{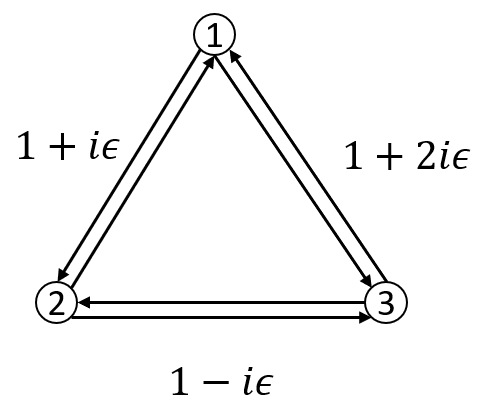}
			\caption{Dual complex unit gain graph $\Phi_A$}
			\label{fig:balanced}
		\end{subfigure}
		\hfill
		\begin{subfigure}[b]{0.45\textwidth}
			\centering
			\includegraphics[width=0.65\linewidth]{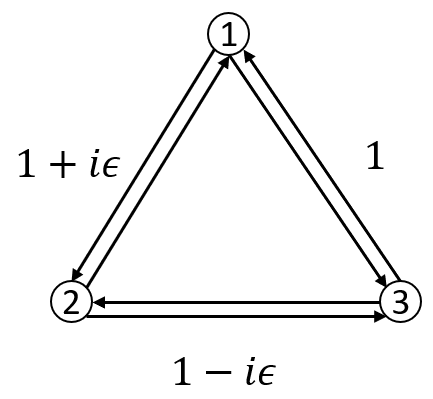}
			\caption{Dual complex unit gain graph $\Phi_B$}
			\label{fig:balanced1}
		\end{subfigure}
		
		\caption{Two dual complex unit gain graphs}
	\end{figure}
	
	Following the same approach, we derive the eigenpairs of the  adjacency matrix of {the} dual complex unit gain  cycle in  {Figure 1~(b).}
	\begin{example}
		Consider  a dual complex unit gain  cycle  in  Figure 1~(b).
		The adjacency   matrix is  given as follows.
		\[B = {\begin{bmatrix}
				0 & 1+i\epsilon & 1 \\
				1-i\epsilon & 0 &  1 -i\epsilon  \\
				1 &  1 +i\epsilon  & 0
			\end{bmatrix}.}	\]
		Its eigenvalues are
		\[\lambda_1=2,\quad \lambda_2=-1,\quad \lambda_3 = -1,\]
		and their corresponding  eigenvectors are
		\[\vx_1=\begin{bmatrix}
			0.5774  + 0.1925i\epsilon \\
			0.5774  - 0.3849i\epsilon \\
			0.5774  + 0.1925i\epsilon \\
		\end{bmatrix},\
		\vx_2=\begin{bmatrix}
			-0.7152 - 0.0055i\epsilon \\
			0.0166  - 0.0055i\epsilon \\
			0.6987  - 0.0055i\epsilon \\
		\end{bmatrix},\]
		and
		\[
		\vx_3=\begin{bmatrix}
			0.6834 + 0.2429i\epsilon \\
			0.7287 + 0.2429i\epsilon \\
			0.0453  + 0.2429i\epsilon \\
		\end{bmatrix},
		\]
		respectively.
		It should be noted that all eigenvalues are real numbers. This is because $\Phi_B$ is balanced {\cite{CLQW24}}.
	\end{example}
	
	\bigskip
	
	\subsection{Eigenvalues of Balanced Dual Unit Gain Cycles}\label{cycles}
	
	We continue to test large-scale dual complex and dual quaternion unit gain  cycles.
	We use the default command {`eig' in MATLAB} to compute {all} eigenvalues and eigenvectors of a complex matrix and the package {`qtmf' }\footnote{\url{https://qtfm.sourceforge.io/}} to compute all eigenpairs of a quaternion matrix.
	{We first generate $n$ unit dual complex or dual quaternion numbers $q_i$, $i=1,\dots,n$.  Then we define the gain of each edge to be $\phi(e_{i,i+1}) = q_i^*q_{i+1}$ for $i=1,\dots,n-1$, and  $\phi(e_{n,1})=q_n^*q_1$. In this way, the corresponding cycles are balanced.}
	For {balanced} dual complex unit gain  {cycles and} dual quaternion unit gain {cycles}, the eigenvalues of the Laplacian {matrices  have} closed from solutions {\cite{CLQW24} as follows,
		\begin{equation}\label{equ:Lap_eig_cycle}
			\sigma_L(\Phi) =  \left\{2-2\cos\left(\frac{\theta+2\pi j}{n}\right):\ j\in\{0,\dots,n-1\}\right\}.
		\end{equation}
		{Let the number of vertices $n\in\{10,20,50,100,200,500\}$. {We} generate  random unit dual {elements} as the gains {and} then compute the eigenpairs by   Algorithm \ref{alg:SMM}}.
		Define the computational {residue} of  {Algorithm \ref{alg:SMM}} {by the $2^R$-norm} of our obtained eigenvalues compared with the closed-form values.
		In Table \ref{table:cycle}, 	we report the CPU time and {residue (RES)}  for computing all eigenvalues and eigenvectors {of dual complex unit gain graphs (DCUGG) and dual quaternion unit gain graphs (DQUGG)}.
		From this table, we  can see  that our proposed method is fast and accurate.
	}
	
	\begin{table}[t]
		\caption{Numerical results for dual complex and dual quaternion unit gain  cycles.}\label{table:cycle}
		\begin{tabular}{c|c|cccccc}
			\toprule
			Method & $n$& 10 & 20 & 50 & 100 & 200 & 500  \\ \hline
			& \multicolumn{6}{c}{DCUGG} \\\hline
			{Algorithm \ref{alg:SMM}} & CPU (s) & 2.15e$-$03 & 2.47e$-$03 & 1.41e$-$02 & 8.09e$-$02 & 8.18e$-$01  &2.49e+01   \\
			& {RES} & 4.23e$-$15 & 6.93e$-$15 & 1.18e$-$14 & 9.95e$-$15  &1.19e$-$14 & 1.60e$-$14 \\
			%			\hline
			%			PM  & CPU (s) & 2.45e$-$02&  8.32e$-$02 & 5.97e$-$01& 3.26e+00  &3.53e+01 & 3.85e+02 \\
			%			& RSE & 1.70e$-$15  &4.29e$-$15 & 3.61e$-$05&  4.20e$-$03 & 2.30e$-$02 & 5.79e$-$01 \\
			\hline
			& \multicolumn{6}{c}{DQUGG} \\\hline
			{Algorithm \ref{alg:SMM}}& CPU (s) & 1.51e$-$02 & 2.77e$-$02&  1.26e$-$01 &  4.15e$-$01  & 3.22e+00 &  8.27e+01  \\
			&  {RES} & 7.43e$-$15 & 1.09e$-$14 & 4.48e$-$14& 8.12e$-$14  & 3.00e$-$13  & 6.99e$-$13   \\
			%			\hline
			%			PM  & CPU (s) &4.40e-01  &2.36e+00&  1.81e+01 &  &&  \\
			%			& RSE & 2.04e-15&  1.94e-15 & 4.13e-04&&&   \\
			\bottomrule
		\end{tabular}
	\end{table}

	%\subsection{Reasonable desired formation}
	\subsection{Balance of Dual Unit Gain Graphs}
	
	%We verify whether the desired relative configuration is reasonable
	{We verify the  dual unit gain graphs are balanced} or not by Theorem \ref{thm:balanced}. We
	{check the first two conditions of Theorem \ref{thm:balanced} manually and define the residue of the third condition by}
	\begin{equation}
		Err = \|Y^*LY- L_G\|_{F^R},
	\end{equation}
	Here, $L$ and $L_G$ are the Laplacian matrices of the gain graph and the underlying graph, respectively.
	{We say the gain graph is balanced if  the first two conditions of Theorem \ref{thm:balanced} hold true and $Err$ is less than a threshold. We set the threshold as as $10^{-8}$ in our numerical experiments.}

	\begin{example}
		We first verify the {balance} of  $\Phi_A$ and $\Phi_B$ in Examples 7.1 and 7.2.
		The Laplacian matrices of $G$, $\Phi_A$, $\Phi_B$ are
		\[L_G =  \begin{bmatrix}
			2 & -1 & -1  \\
			-1  & 2 &  -1 \\
			-1  &  -1  & 2
		\end{bmatrix},\quad L_A =  \begin{bmatrix}
			2 & -1-i\epsilon & -1+2i\epsilon \\
			-1+i\epsilon & 2 &  -1+i\epsilon  \\
			-1-2i\epsilon &  -1 -i\epsilon  & 2
		\end{bmatrix},	\]
		\[L_B =  \begin{bmatrix}
			2 & -1-i\epsilon & -1  \\
			-1+i\epsilon & 2 &  -1+i\epsilon  \\
			-1 &  -1 -i\epsilon  & 2
		\end{bmatrix},\]
		respectively.
		The eigenvalues of $L_A$ are $\{0,3,3\}$, the eigenvector corresponding to the zero eigenvalue is
		\[\vx_A = [0.5774-0.1925i \epsilon, 0.5774- 0.3849i\epsilon, 0.5774+0.5774i\epsilon]^\top.\]
		From this, we conclude that the first two conditions of Theorem \ref{thm:balanced} hold true and  $Err_A = 1.6330$. Hence, $\Phi_A$ is not balanced.
		The eigenvalues of $L_B$ are $\{0,3,3\}$, the eigenvector corresponding to the zero eigenvalue is
		\[\vx_B = [0.5774+0.1925i \epsilon, 0.5774- 0.3849i\epsilon, 0.5774+0.1925i\epsilon]^\top.\]
		From this, we see that  the first two conditions of Theorem \ref{thm:balanced} hold true and $Err_B = 2.22e-16$. Hence, $\Phi_B$ is   balanced. 	
	\end{example}

		\begin{table}[t]
		\centering
		\caption{Numerical results for verifying the {balance} of the dual   unit gain  cycles.}\label{table:cycle_balance}
		\begin{tabular}{c|cccccc}
			\toprule
			$n$& 10 & 20 & 50 & 100 & 200 & 500  \\ \hline
			\multicolumn{6}{c}{DCUGG} \\\hline
			CPU (s) & 7.78e$-$04 & 1.79e$-$03 & 1.36e$-$02 & 8.15e$-$03 & 5.54e$-$02 & 1.96e$-$01   \\
			Err &6.92e$-$15 & 3.02e$-$14 & 4.04e$-$14 & 1.69e$-$13 & 1.34e$-$13 & 1.36e$-$12\\
			\hline
			\multicolumn{6}{c}{DQUGG} \\\hline
			CPU (s) & 1.20e$-$02 & 2.60e$-$02 & 7.47e$-$02 & 2.57e$-$01 & 1.28e+00 & 4.13e+01  \\
			Err & 9.81e$-$15 & 2.35e$-$14 & 4.96e$-$14 & 1.20e$-$13 & 4.50e$-$13 & 1.47e$-$12  \\
			\bottomrule
		\end{tabular}
	\end{table}

	We continue to test the cycles in Section~\ref{cycles}.
	By Table~\ref{table:cycle_balance}, we see that all results can be obtained in $5$ seconds, and the residue is less than $10^{-11}$ for all examples. {These results show that we can verify the balance of the dual unit gain graphs efficiently.}
	
\bigskip
	
	{\section{Concluding Remarks}
	
	In this paper, we have made the following contributions.
	
	\begin{itemize}
	
	\item We studied dual number symmetric matrices, dual complex Hermitian matrices and dual quaternion Hermitian matrices in a {\bf unified frame} as dual Hermitian matrices.  This avoided unnecessary repetitions.
	
	\item We proposed a {\bf practical method} - The Supplement Matrix Method, for computing eigenvalues of a dual Hermitian matrix.
	
	\item We raised a {\bf meaningful application problem} - The Relative Configuration Problem,  in multi-agent formation control.
	
	\item We explored a {\bf cross-disciplinary approach} to solve the above problem.   This approach combines the spectral theory of dual Hermitian matrices, and the unit gain graph theory.  While the unit gain graph theory is well-developed in spectral graph area, what we used here are about dual quaternion unit gain graphs.  This is the first discussion on dual quaternion unit gain graphs in the literature.
Finally, the supplement matrix method was used in this approach.
	
	\end{itemize}

We will continue on this path, namely exploring on the application oriented research.
	
	%Future Works:
	%\begin{itemize}
	
	%\item Multi-Agent Formation Control.
	
	%\item Cross-Disciplinary Theory.
	
	%\item	Low Rank Approximation of a Dual Matrix.
	
	%\item	The Dual Least Squares Problem and Generalized Inverses of Dual Matrices.

	%\end{itemize}
	}

	%\section*{Compliance with ethical standards}
	\bigskip
	
	%{\bf Conflicts of Interest} The author declares no conflict of interest.

	% \vspace{100pt}

\end{document}